\documentclass[12pt]{amsart}

\usepackage[top=.85 in,bottom=.85 in, left=.85 in, right= .85 in]{geometry}

\newcommand{\NEF}{natural exponential family (NEF)\renewcommand{\NEF}{NEF}}

\usepackage{fancybox,amssymb,amsthm,amsfonts,dsfont,float}

 \floatstyle{ruled}
\restylefloat{table}

 \newcommand{\ceta}{\alpha}
 \newcommand{\ctheta}{\beta}

\title{Stitching pairs of L\'evy processes into harnesses}

\date{Created: July 21, 2010. Printed \today. File \jobname.tex}

\author{
W{\l}odek  Bryc
}

\address{
Department of Mathematics,
University of Cincinnati,
PO Box 210025,
Cincinnati, OH 45221--0025, USA}
\email{Wlodzimierz.Bryc@UC.edu}

\author{Jacek Weso{\l}owski}
\address{ Faculty of Mathematics and Information Science\\
Warsaw University of Technology\\ pl. Politechniki 1\\ 00-661
Warszawa, Poland}
\email{wesolo@alpha.mini.pw.edu.pl}


\numberwithin{equation}{section}

\newcommand{\be}{\begin{equation}}
\newcommand{\ee}{\end{equation}}
      \newtheorem{theorem}{Theorem}[section]
       \newtheorem{proposition}[theorem]{Proposition}
       
       \newtheorem{lemma}[theorem]{Lemma}
              \newtheorem{claim}[theorem]{Claim}
       \newtheorem{remark}{Remark}[section]
\theoremstyle{definition}
\newtheorem{definition}{Definition}[section]

\newtheorem{assumption}{Assumption}
\newtheorem{example}{Example}[section]

\def\RR{{\mathds R}}

\def\Z{{\mathbf Z}}

\def\<{\langle}
\def\>{\rangle}

\def\E{{\mathbb E}}

\def\var{\mathbb{V}\mathrm{ar}}
\def\cov{\mathbb{C}\mathrm{ov}}

\def\Var{\var}
 \def\V{\var}
\def\Cov{\cov}

\newcommand{\calF}{\mathcal{F}}

\newcommand{\TT}{\mathcal{T}}

\newcommand{\1}{\mathds{1}}
\newcommand{\la}{\lambda}

\begin{document}

\begin{abstract}
We consider natural exponential families of L\'evy processes with randomized parameter.  Such processes are Markov, and under suitable assumptions, pairs of  such processes with shared randomization  can be ``stitched together" into a single harness. The stitching consists of deterministic reparametrization of the time for both processes, so that they run on adjacent time intervals, and of the choice of the appropriate law at the boundary.

Processes in the L\'evy-Meixner class have an additional property that they are quadratic harnesses, and in this case stitching constructions produce quadratic harnesses on $[0,\infty)$.
\end{abstract}

\maketitle

  \section{Introduction}
  Two ad hoc constructions of  quadratic harnesses from transition probabilities led to Markov processes that  naturally split into a pair of conditionally independent Poisson processes (\cite[Proposition 4.1]{Bryc-Wesolowski-05}) or a pair of conditionally independent negative binomial processes (\cite[Proposition 5.1]{Maja:2009}) with random parameters.
 This paper proceeds in the opposite direction, producing   harnesses and quadratic harnesses directly from a stitching construction rather than from transition probabilities.
The appropriate general setting for this approach is to consider natural exponential families of L\'evy processes 
which come with a parameter ready for randomization, and where one can draw on powerful results from \cite{Diaconis:1979} to determine randomizations responsible for martingale property.
    \subsection{Exponential families of L\'evy processes}\label{Sect:MCRMP}
We recall the construction of the \NEF\ of a  L\'evy process from \cite[Chapter 2]{Kuchler:1997}. We consider a L\'evy process $(\xi_t)$ on a probability space $(\Omega,\calF,P)$ with the natural (past)  filtration $\calF_t$. Since we will be working only with finite-dimensional distributions, we  take $\Omega=\{f:[0,\infty)\to \RR\}$ with
$\calF_t$ generated by Borel sets that do not depend on the trajectories $f$ after time $t$.
Let $L(\theta)=\E\exp(\theta \xi_1)$ and $\kappa(\theta)=\log L(\theta)$ be well defined for $\theta$ in an open interval $(\theta_0,\theta_1)\subset\RR$ and let $g(dx)$ be the law of $\xi_1$.
We set $\kappa(\theta)=\infty$ outside of $(\theta_0,\theta_1)$.

The \NEF\ generated by the law $P$ of process $(\xi_t)$ is a measure $P_\theta$ on the $\sigma$-field generated by $\bigcup_{t>0}\calF_t$ such that
\begin{equation}\label{P_theta}
P_\theta \big|_{\calF_t}(df)=\exp(\theta f(t)-t\kappa(\theta))P\big|_{\calF_t}(df).
\end{equation}
Note that  \eqref{P_theta} is just a prescription that generates a consistent family of finite dimensional distributions, so such a measure exists on $\Omega=\RR^{[0,\infty)}$. It will be convenient to denote by $(X_t^{(\theta)})$ a L\'evy  process on some abstract probability space, with the same finite dimensional distributions as the process  $\Omega\ni f  \mapsto f(t)\in\RR$  under the law $P_\theta$. %
It is well known (\cite[Section 2.3]{Jorgensen:1997}) that $\E(X_t^{(\theta)})=t\kappa'(\theta)$ and $\Var(X_t^{(\theta)})=t\kappa''(\theta)$.

\subsection{Stitching construction}\label{Sect:SCG}
We are interested in processes with random $\theta$. To this end, we choose a probability measure $h(d\theta)$ on Borel subsets of $(\theta_0,\theta_1)$ and introduce   probability measure $Q(df)=\int P_\theta(df) h(d\theta)$. It will be convenient to denote by  $(Y_t)$ any process on some abstract probability space with the same finite dimensional distributions as  the coordinate process $\Omega\ni f\mapsto f(t)$ under $Q$.
Denoting by $\Theta$ the random variable with law $h(d\theta)$, alternatively we can view $(Y_t)$ as a process 
which conditionally on $\Theta=\theta$ has the same finite dimensional laws as the process $(X_t^{(\theta)})$.
An example of such situation is randomized hyperbolic secant process (Example \ref{P1-Meix} below) studied in \cite{grigelionis1999processes}.

 Our  goal is to   stitch together   pairs of randomized conditionally independent  L\'evy processes into a single process.
  To do so,  we  consider a pair  of $\Theta$-conditionally independent Markov processes  $(Y_t)$ and $(Y_t')$ which conditionally on $\Theta$ have the same  laws.  %

  The stitched process, with auxiliary parameters $r,v>0$ and $p\in\RR$, is given by
\begin{equation}\label{Exp-Z}
Z_t=\begin{cases}
 \tfrac{1-t}{rv}Y_{rt/(1-t)}-t\tfrac{p}{rv},& 0\leq t <1,\\
 \\
\tfrac{1}{v} (\kappa'(\Theta)-p/r),& t=1,\\
\\
 \tfrac{t-1}{rv}Y_{r/(t-1)}'-\tfrac{p}{rv}, & t>1.
\end{cases}
\end{equation}

We remark that parameter $v$ is used in \eqref{Exp-Z} solely for standardization: in the square-integrable case we use it so that $Z_1$ has unit variance.
 To motivate the use of parameters $p,r$, suppose that $(Y_t+p)/(t+r)$ is a martingale so that $\E\left(\kappa'(\Theta)\right)=p/r$. Then it is not difficult to verify that $(Z_t)$
 is a  martingale with mean zero.

We note that $(Z_t)_{t>1}$ is the time inverse of $(Z_t)_{0<t<1}$. This observation reduces the number of cases to be considered in some proofs.
It is also convenient to observe that $(Z_t)_{t>0}$ is a Markov process. This follows by construction from Markov property of   $(Y_t)$ and from $\Theta$-conditional independence of $(Y_t)$ and  $(Y_t')$.

The paper is organized as follows. In Section \ref{Sect:Cond:str} we investigate 
conditional properties
of process $(Y_t)$.
In Section \ref{Sect:Harness} we show that  under appropriate randomization $(Z_t)$ is a harness. In Section \ref{Sect:Meinxer-Examples} we discuss  L\'evy-Mexiner processes   and their randomization. Our main result in Section \ref{Sect:SC} states that stitching of
  a pair of L\'evy-Meixner processes gives a quadratic harness. The proof is in Section \ref{Sect:Proofs}.

\section{Conditional properties of process $(Y_t)$}\label{Sect:Cond:str}

\begin{proposition}\label{L-Markov}
$(Y_t)$ is a Markov process with $Y_0=0$ and
with transition probabilities
\begin{equation}\label{Exp-transitions}
P_{s,t}(x,dy)=\frac{H(t,y)}{H(s,x)} g_{t|s}(dy|x), \; s<t,
\end{equation}
 where $g_{t|s}(dy|x)$ denotes the regular version of the conditional law of $\xi_{t}$ given $\xi_{s}$, and
\begin{equation} \label{H-exp}
H(t,x)=\int_{\theta_0}^{\theta_1} e^{\theta x -t\kappa(\theta)} h(d\theta).
\end{equation}

Furthermore, the transition probabilities for $(Y_t)$  in reversed time are the same as for the L\'evy process $(\xi_t)$,
\begin{equation}\label{Inverse-trans}
P_{t,s}(y,dx)=g_{s|t}(dx|y),\; s<t.
\end{equation}
 \end{proposition}
\begin{proof}
 For $s_1<s_2<\dots<s_n$, let $g_{s_1,\dots,s_n}(dx_1,dx_2,\dots dx_n)$ denote the finite dimensional distribution of process $(\xi_t)$. From \eqref{P_theta} it follows that the finite-dimensional laws of process $(Y_t)$ are
 \begin{equation}\label{Exp-Markov}
G_{s_1,\dots,s_n}(dx_1,dx_2,\dots dx_n)= H(s_n,x_n) g_{s_1,\dots,s_n}(dx_1,dx_2,\dots dx_n),
 \end{equation}
where $H(t,x)$ is defined in \eqref{H-exp}.
Since $(\xi_t)$ is a Markov process, for any bounded measurable functions $\varphi,\psi$ we have
\begin{multline*}
\int  \varphi(x_n) \psi(x_1,\dots,x_{n-1})G_{s_1,\dots,s_n}(dx_1,dx_2,\dots dx_n)
\\=
\int_{\RR^n} \varphi(x_n)\psi(x_1,\dots,x_{n-1})H(s_n,x_n) g_{s_1,\dots,s_n}(dx_1,dx_2,\dots dx_n)
\\
=\int_{\RR^{n-1}} \psi(x_1,\dots,x_{n-1})
\left(\int_\RR \varphi(x_n)\frac{H(s_n,x_n)}{H(s_{n-1},x_{n-1})}g_{s_n|s_{n-1}}(dx_n|x_{n-1})\right)
 G_{s_1,\dots,s_{n-1}}(dx_1,dx_2,\dots dx_{n-1}).
\end{multline*}
So
$$
\E(\varphi(Y_{s_n})|Y_{s_1},\dots,Y_{s_{n-1}})=
\int_\RR \varphi(y)\frac{H(s_n,y)}{H(s_{n-1},Y_{s_{n-1}})}g_{s_n|s_{n-1}}(dy|Y_{s_{n-1}})
$$
and
 $(Y_t)$ is Markov with transition probabilities \eqref{Exp-transitions}.

To verify \eqref{Inverse-trans}, note that
\begin{multline*}
\int  \varphi(y) \psi(x)G_{s,t}(dx, dy)=
\int_{\RR^2} \varphi(y) \psi(x)  H(t,y) g_{s,t}(dx,dy)
\\
=\int_{\RR} \varphi(y)
\left(\int_\RR \psi(x)  g_{s|t}(dx|y)\right)
H(t,y)g_{t}(dy).
\end{multline*}
Now the result follows from \eqref{Exp-Markov}.
\end{proof}

The following  result supplements Proposition \ref{L-Markov}.

\begin{lemma}\label{L-two-sided}
Fix $s<t<u$. The two-sided conditional laws of $Y_t$ given $Y_s,Y_u$ are the same as the two-sided conditional laws  $\xi_t|\xi_s,\xi_u$ of the L\'evy process $(\xi_t)$.
\end{lemma}
 \begin{proof}
   From \eqref{Exp-transitions} and \eqref{Exp-Markov}, the joint law of $Y_s,Y_t,Y_u$ is given by
\begin{equation}\label{TO}
P_{0,s}(0,dx)P_{s,t}(x,dy)P_{t,u}(y,dz)=H(u,z) g_{s,t,u}(dx,dy,dz),
\end{equation}
where $g_{s,t,u}(dx,dy,dz)$ is the joint distribution of $(\xi_s,\xi_t,\xi_u)$.
Fix a bounded measurable function $\varphi(x,z)$.
If
$\E(e^{i\alpha \xi_t}|\xi_s,\xi_u)=f(\xi_s,\xi_u)$ is the conditional characteristic  function for the L\'evy process, then,
since    $\varphi(x,z) H(u,z)$ is integrable with
 respect to $g_{s,t,u}(dx,dy,dz)$,  we have
$$
\int  e^{i\alpha y}\varphi(x,z) H(u,z) g_{s,t,u}(dx,dy,dz)= \int f(x,z)  \varphi(x,z) H(u,z) g_{s,t,u}(dx,dy,dz).
$$
So by \eqref{TO}, we get
\begin{equation*}
  \int  e^{i\alpha y}\varphi(x,z) P_{0,s}(0,dx)P_{s,t}(x,dy)P_{t,u}(y,dz)
= \int  f(x,z) \varphi(x,z) P_{0,s}(0,dx)P_{s,t}(x,dy)P_{t,u}(y,dz),
\end{equation*}
which shows that the conditional characteristic function  $\E(e^{i\alpha Y_t}|Y_s,Y_u)$ is the same as the conditional characteristic function $\E(e^{i\alpha \xi_t}|\xi_s,\xi_u)$.
 \end{proof}

Next, we use martingale property to determine the class of randomizations that will be used for stitching.   
 \begin{proposition}\label{P1-Exp} Denote by $g(dx)$ the law of $\xi_1$. We assume that
 $\xi_1$ is integrable and that  $\kappa'(\theta)$ is $h(d\theta)$- integrable. Suppose that one of the following conditions holds:
 \begin{enumerate}
 \item The support of $g(dx)$ contains an open interval;
 \item  The support of $g(dx)$ is non-negative integers, and  $\theta_1<\infty$;
 \item  $\kappa(\theta)=e^{\theta}$, so that $g(dx)$ is the Poisson measure  on non-negative integers.\end{enumerate}
 If there are constants $p\in\RR$, $r>0$ such that $(Y_t+p)/(t+r)$ is a martingale with respect to its natural filtration, $t\geq 0$,
 then
\begin{equation}\label{Exponential-law}
h(d\theta)= C e^{p\theta-r \kappa(\theta)} 1_{(\theta_0,\theta_1)}(\theta)d\theta.
\end{equation}

\end{proposition}
\begin{proof}  In fact, we only use the following simple consequence of the assumed martingale property: there is a pair $0<s<t$ such that
\begin{equation}\label{Exp-DY}
\E(Y_t-Y_s|Y_s)=(t-s)\frac{Y_s+p}{s+r}.
\end{equation}
Since the conditional law of $Y_t-Y_s|Y_s,\Theta$ is $e^{\Theta z - (t-s) \kappa(\Theta)}g_{t-s}(dz)$, we have
\begin{equation}\label{1.8}
\E(Y_t-Y_s|Y_s,\Theta)=(t-s)\kappa'(\Theta).
\end{equation}
 This implies that
\begin{equation}\label{Exp-YYY}
\E(\kappa'(\Theta)|Y_s)=\frac{Y_s+p}{s+r}.
\end{equation}
The joint law of $(Y_s,\Theta)$ is
\begin{equation}
  \label{niezanumarowany}
  e^{\theta y -s\kappa(\theta)}g_s(dy)h(d\theta),
\end{equation}
so this is a setting analyzed in \cite{Diaconis:1979}.
It is clear that for an infinitely divisible family of laws $\{g_t(dx):t>0\}$, the support of $g_s(dx)$ inherits the property of $g(dx)$ assumed in (i), (ii) or (iii).
The result follows from \cite[Theorem 3]{Diaconis:1979} in case (i), from \cite[Theorem 4]{Diaconis:1979} in case (ii), and from \cite{Johnson:1957} in case (iii).

\end{proof}

\begin{remark}
  The  martingale characterization of the law $h(d\theta)$ in case (iii) is related to \cite{nekrutkin2007martingale}.
\end{remark}

To establishes properties of the stitching construction  when the law of $\Theta$ is \eqref{Exponential-law}, we rely on additional technical assumptions on the L\'evy process $(\xi_t)$.
 \begin{assumption} We assume that $p\in\RR$ and $r>0$ are such that
for all $x$ in the support of $g(dx)$
\begin{equation}\label{Asume1}
\lim_{\theta\to \theta_0^+} e^{(p+x)\theta - r \kappa(\theta)}=\lim_{\theta\to \theta_1^-} e^{(p+x)\theta - r \kappa(\theta)}=0,
\end{equation}
\begin{equation}\label{Asume2}
\lim_{\theta\to \theta_0^+}\kappa'(\theta) e^{(p+x)\theta - r \kappa(\theta)}=\lim_{\theta\to \theta_1^-} \kappa'(\theta) e^{(p+x)\theta - r \kappa(\theta)}=0.
\end{equation}
\end{assumption}
(In Section \ref{Sect:Meinxer-Examples} we give  examples of processes $(\xi_t)$ that satisfy these assumptions.)

The following is a converse to Proposition \ref{P1-Exp} under additional assumption  \eqref{Asume1}.
\begin{proposition}\label{P-conv}Suppose a L\'evy process $(\xi_t)$ is integrable, and that  the law of $\Theta$ is \eqref{Exponential-law} with parameters $p\in\RR$ and $r>0$, such that \eqref{Asume1} holds, and that $\kappa'(\Theta)$ is integrable.
Then   process  $(Y_t+p)/(t+r)$ is a martingale in  the natural filtration.
\end{proposition}
\begin{proof}
 By Markov property, see Proposition \ref{L-Markov},
we want to show that
$$
(s+r)\E(Y_t-Y_s|Y_s)=(t-s)(Y_s+p).
$$
Since \eqref{1.8} holds, using \eqref{Exponential-law} and \eqref{niezanumarowany} we see that
it suffices to verify that for a bounded measurable $\varphi$, we have
$$
\iint \varphi(y) (y+p-(s+r)\kappa'(\theta))e^{(y+p)\theta - (s+r)\kappa(\theta)} d\theta g_s(dy)=0.
$$
Since $\xi_s, k'(\Theta)$ are integrable by assumption, we can switch to iterated integrals
\begin{equation}\label{E(K')}
\int_\RR \int_{\theta_0}^{\theta_1} (y+p-(s+r)\kappa'(\theta))e^{(y+p)\theta - (s+r)\kappa(\theta) }d\theta \varphi(y) g_s(dy).
\end{equation}
Under  \eqref{Asume1}, the inner integral is
$$\lim_{(u,v)\to(\theta_0,\theta_1)}
e^{(p+y) v  -(r+s)\kappa(v)}-e^{(p+y) u  -(r+s)\kappa(u)}=0-0.$$
Once this holds, we get \eqref{Exp-YYY}, and then \eqref{Exp-DY} which gives martingale as $(Y_t)$ is Markov.
\end{proof}
In particular, since  %
a martingale must have constant mean, we get $\E(Y_t)=t p/r$. So if \eqref{Asume1} holds, then from  \eqref{1.8} we get
\begin{equation}\label{p/r}
\E(\kappa'(\Theta))=\int_{\theta_0}^{\theta_1} \kappa'(\theta)h(d\theta)=p/r.
\end{equation}

\section{Harness property}\label{Sect:Harness}
The following definition is a Markov version of the well-known concept of a harness, see
\cite{Hammersley,Mansuy-Yor-05}.
\begin{definition}\label{Def_0} Let $\TT=(T_0,T_1)\subset(0,\infty)$. A Markov process $(Z_t)$ is a harness on $\TT$, if  for every $s,t,u\in \TT$ with $s<t<u$,
\begin{equation}\label{LR}
\E(Z_t|Z_s,Z_u)=\frac{u-t}{u-s}Z_s+\frac{t-s}{u-s}Z_u.
\end{equation}

\end{definition}

We now show that the randomization laws identified in Proposition \ref{P1-Exp} yield harness.
  \begin{proposition}\label{P2.3} Suppose a L\'evy process $(\xi_t)$ is integrable and that  the law of $\Theta$ is \eqref{Exponential-law} with some parameters $r>0$, $p\in\RR$, such that \eqref{Asume1} holds and $\kappa'(\Theta)$ is integrable.
Then \eqref{Exp-Z}  defines a harness on $(0,\infty)$.
\end{proposition}

\begin{proof}
We only need to verify \eqref{LR} for $s<t<u<1$ and for $s<t=1<u$.
Indeed, if we have these two cases, then the remaining cases are handled as follows:  the  case $1<s<t<u$ is the time-inversion of $0<s<t<u<1$.
In the case $0<s<t<1<u$  by Markov property
$\E(Z_t|Z_s,Z_u)=\E(E(Z_t|Z_s,Z_1)|Z_s,Z_u)=\frac{1-t}{1-s}Z_s+\frac{t-s}{1-s} E(Z_1|Z_s,Z_u)$.  The other case $0<s<1<t<u$ is handled similarly (or by time inversion).
Finally, the cases  $1=s<t<u$  and $s<t<u=1$ are the limits of cases $0<s<1<t<u$ and $0<s<t<1<u$, respectively.

To prove  \eqref{LR} for or $s<t<u<1$, denote %
\begin{equation}\label{stu}
s'= \frac{rs}{1-s},\;  t'= \frac{rt}{1-t},\; u'=\frac{ ru}{1-u}.
\end{equation}

Since it is known that all integrable L\'evy processes are harnesses, see e.g. \cite[(2.8)]{Jacod-Protter-88}, by Lemma \ref{L-two-sided} we get
\begin{equation}\label{Z-LR}
\E(Z_t|Z_s,Z_u)=\frac{1-t}{rv}\left(\frac{u'-t'}{u'-s'}Y_{s'}+\frac{t'-s'}{u'-s'}Y_{u'}\right).
\end{equation}
A calculation shows that
$$
\frac{u'-t'}{u'-s'}Y_{s'}+\frac{t'-s'}{u'-s'}Y_{u'}=\frac{p t }{1-t}+rv\frac{(u-t) Z_s+(t-s)  Z_u}{(1-t) (u-s)}
,
$$
so the right hand side of \eqref{Z-LR} simplifies to the right hand side of \eqref{LR}. (This part of the proof does not rely on \eqref{Exponential-law}.)

To prove  \eqref{LR} for or $s<t=1<u$,   denote
\begin{equation}\label{rprmt}
s'=\frac{rs}{1-s} ,\; u'=\frac{ r}{u-1}.
\end{equation}
The joint distribution  $Z_s,Z_1,Z_u$ (see \eqref{Exp-Z}) is determined from the joint distribution of $Y_{s'},\Theta,Y'_{u'}$ which by conditional independence is given by
\begin{equation}\label{mumu}
\mu(dx,d\theta,dz)=C e^{(x+z+p)\theta-(r+s'+u')\kappa(\theta)}d\theta g_{s'}(dx)g_{u'}(dz).
\end{equation}
To verify harness property, we show that
\begin{equation}\label{mu-phi0}
\E\left(\kappa'(\Theta)|Y_{s'},Y'_{u'}\right)=\frac{Y_{s'}+Y'_{u'}+p}{u'+s'+r}.
\end{equation}
 Equivalently, we show that for any bounded measurable function $\varphi(x,z)$,
\begin{equation}\label{mu-phi}
\iiint \varphi(x,z)( x+z+p-(u'+s'+r)\kappa'(\theta)) \mu(dx,d\theta,dz)=0.
\end{equation}
Since we assume integrability, to prove \eqref{mu-phi} we  rewrite the triple integral into the iterated integrals, with the inner integral with respect to $\theta$. The inner integral is
\begin{multline*}
\int_{\theta_0}^{\theta_1}( x+z+p-(u'+s'+r)\kappa'(\theta)) e^{ (x+z+p)\theta-(u'+s'+r)\kappa(\theta)}d\theta
= e^{ (x+z+p)\theta-(u'+s'+r)\kappa(\theta)}\Big|_{\theta=\theta_0}^{\theta=\theta_1}=0.
\end{multline*}
(Here we use \eqref{Asume1}, noting that $x+z$ is in the support of $g(dx)$ by infinite divisibility.) This shows that $E(Z_1|Z_s,Z_u)$ is a linear function of $Z_s,Z_u$.
A calculation shows that
\begin{equation}\label{Y2ZZ}
\frac{Y_{s'}+Y'_{u'}+p}{u'+s'+r}=\frac{p}{r}+v\frac{Z_u-Z_s}{u-s}+v\frac{uZ_s-sZ_u}{u-s},
\end{equation}
so \eqref{LR} follows.
\end{proof}

\section{Randomizations of L\'evy-Meixner processes}\label{Sect:Meinxer-Examples}

 Recall that the variance function of a \NEF\ (\cite{Jorgensen:1997,letac1992lectures}) is a function $V$ such that
\begin{equation}\label{Var-Funct}
\kappa''(\theta)=V(\kappa'(\theta)).
\end{equation}
It is known (\cite[Theorem 2.11]{Jorgensen:1997}) that $V$ determines $\kappa$ uniquely.  In this section we consider   L\'evy processes $(\xi_t)$ in the Meixner class \cite{schoutens2000stochastic}. Each such process  generates a \NEF\  with a quadratic variance function
\begin{equation}\label{QV}
V(m)=am^2+bm+c,\; a\geq 0,
\end{equation}
 see \cite{Morris:1982}.
These processes are square-integrable, and by checking each case we verify that they satisfy (\ref{Asume1}-\ref{Asume2}), provided $r>a$. (Restrictions on $p$ vary per case.)

We now list each of the five cases  of L\'evy processes corresponding to exponential families with quadratic variance function $V$.

\begin{example}\label{P1-Wiener}
Let $(\xi_t)_{t\ge 0}$ be the Wiener process. Then $\kappa(\theta)=\theta^2/2$ and $\kappa'(\theta)=\theta$. The variance function \eqref{Var-Funct} is  $V(m)=1$.

For $\theta\in(\theta_0,\theta_1)=(-\infty,\infty)$, process $(X_t^{(\theta)})$ has univariate densities $\sim \exp(\frac{(x-t\theta)^2}{2t})$, so it has a very simple representation $X_t^{(\theta)}=\xi_t-t\theta$.

Next, consider random $\Theta$, and (Markov) process  $(Y_t)_{t\ge 0}$.
  If $\Theta$ is integrable and there are constants $p\in\RR$, $r>0$ such that $(Y_t+p)/(t+r)$ is a martingale with respect to natural filtration, then formula \eqref{Exponential-law} implies that
  $\Theta$ is normal with mean $p/r$ and variance $1/r$.
 Then the moments are:
\begin{equation*}
\E(\kappa'(\Theta))=\tfrac{p}{r},\; \Var(\kappa'(\Theta))=\tfrac{1}{r}.
\end{equation*}
It is easy to see that \eqref{Asume1},  \eqref{Asume2} hold for $p\in\RR$, $r>0$.
\end{example}

The following example is closely related to \cite{Johnson:1957} and to \cite[Proposition 1]{nekrutkin2007martingale}.
\begin{example}\label{P1-Poiss}
Let $(\xi_t)_{t\geq 0}$ be the Poisson process with parameter $\la=1$. Then $\kappa(\theta)=e^{\theta}-1$, $\kappa'(\theta)=e^\theta$. The variance function \eqref{Var-Funct} is  $V(m)=m$.

For $\theta\in(\theta_0,\theta_1)=(-\infty,\infty)$, process $(X^{(\theta)}_t)$ is a Poisson process with parameter $\la=e^\theta$, so it has a very simple representation $X_t=\xi_{te^\theta}$.

Next, consider random $\Theta$, and (Markov) process  $(Y_t)_{t\ge 0}$.
If $e^\Theta$ is integrable and there are constants $p>0$, $r>0$ such that $(Y_t+p)/(t+r)$ is a martingale with respect to its natural filtration, $t\geq 0$, then the law of $\Theta$ is
\begin{equation}\label{HP}
h(d\theta)=C \exp(p \theta - r (e^\theta-1)) d\theta.
\end{equation}

A more natural randomization is $\Lambda=\kappa'(\Theta)=\exp(\Theta)$,
 then    $\Lambda$ has gamma $G(p,r)$ law,
\begin{equation*}\label{Poiss-La}
h(d\la)=\frac{r^p}{\Gamma(p)}\la^{p-1}e^{-r\la}\1_{(0,\infty)}(\la)d\la\,.
\end{equation*}
Under \eqref{HP},
\begin{equation*}
\E(\kappa'(\Theta))=\tfrac{p}{r},\; \Var(\kappa'(\Theta))=\tfrac{p}{r^2}.
\end{equation*}
It is easy to check that  \eqref{Asume1},  \eqref{Asume2} hold for $p>0$, $r>0$.
\end{example}

\begin{example}\label{P1-gamma}
Let $(\xi_t)_{t\ge 0}$ be the standard gamma process, that is a L\'evy process for which $\xi_1$ is exponential with mean 1. Then $\kappa(\theta)=-\log(1-\theta)$, $\kappa'(\theta)=1/(1-\theta)$.
The variance function \eqref{Var-Funct} is  $V(m)=m^2$.

For $\theta\in(\theta_0,\theta_1)=(-\infty,1)$, random variable $X_t^{(\theta)}$ has the gamma law with density proportional to $ x^{t-1}e^{-(1-\theta)x}$ on $(0,\infty)$ so the process has simple representation $X_t=(1-\theta)\xi_t$.

Let $\Theta\in(-\infty,1)$ be a random variable and consider (Markov) process $(Y_t)$.
 Suppose that  $\kappa'(\Theta)=1/(1-\Theta)$ is integrable and  there are constants
$p,r>0$ such that $(Y_t+p)/(t+r)$ is a martingale with respect to natural filtration, $t\geq0$. Then
from \eqref{Exponential-law}, we see that the law of $\Theta$ is
\begin{equation}\label{HG}
C (1-\theta)^r e^{-p(1-\theta)}1_{(-\infty,1)}(\theta)d\theta.
\end{equation}

So  $1-\Theta$ has   gamma $G(r+1,p)$ law i.e., and a more natural parametrization is $W=\kappa'(\Theta)=1/(1-\Theta)$ has density
\begin{equation*}\label{gamma-W}
h(d w)=   \frac{p^{r+1}}{\Gamma(r+1)}\frac{\exp(-p/w)}{w^{r+2}}\1_{(0,\infty)}(w) dw.
\end{equation*}
Under \eqref{HG}, if $r>1$ then
\begin{equation*}
\E(\kappa'(\Theta))=\tfrac{p}{r},\; \Var(\kappa'(\Theta))=\tfrac{p^2}{r^2(r-1)}.
\end{equation*}
It is easy to see that \eqref{Asume1},  \eqref{Asume2} hold for $p>0$, $r>1$.
\end{example}

\begin{example}\label{P1-Pasc}
Let $(\xi_t)_{t\ge 0}$ be the negative binomial process, that is a L\'evy process for which $\xi_t$ is Negative Binomial $NB(q,t)$, i.e.
\begin{equation*}
  \label{NB-def}P(\xi_t=k)=\frac{\Gamma(t+k)}{\Gamma(t)k!}(1-q)^tq^k, \; k=0,1,\dots
\end{equation*}
 (Here $t>0$ and $0<q<1$.)
Then $\kappa(\theta)=\log (1-q)-\log (1-q e^\theta)$ and $\kappa'(\theta)=\frac{qe^\theta}{1-qe^\theta}$.
The variance function \eqref{Var-Funct} is  $V(m)=m^2+m$.

For $\theta\in(\theta_0,\theta_1)=(-\infty, -\log q)$  the natural exponential L\'evy process $(X_t^{(\theta)})$ is negative binomial with parameter $q$ replaced by $q e^\theta$.

Let $\Theta$ be a random variable with values in $(\theta_0,\theta_1)$, and let $(Y_t)$ be the corresponding Markov process.
  Suppose that $\kappa'(\Theta)$ is integrable, and there are constants
$p>0$, $r>0$ such that $(Y_t+p)/(t+r)$ is a martingale in its natural filtration for $t>0$.
Then
\begin{equation}\label{HNB}
h(d\theta)=c e^{\theta p}(1-qe^\theta)^r 1_{(-\infty, -\log q)}(\theta)d\theta.
\end{equation}
A more natural parametrization is
$\Pi=qe^\Theta$,
which has beta $B_I(a,b)$ law,
\begin{equation*}\label{Pascal-Pi}
h(dx)=\frac{\Gamma(a+b)}{\Gamma(a)\Gamma(b)}  x^{a-1}(1-x)^{b-1}  \1_{(0,1)}(x)dx,
\end{equation*}
with parameters $a=p$ and $b=r+1$.
Under \eqref{HNB}, if $r>1$ then
\begin{equation*}
\E(\kappa'(\Theta))=\tfrac{p}{r},\; \Var(\kappa'(\Theta))=\tfrac{p(p+r)}{r^2(r-1)}.
\end{equation*}
It is easy to see that \eqref{Asume1},  \eqref{Asume2} hold for $p>0$, $r>1$.
\end{example}

\begin{remark} When  $\Theta$  has  density \eqref{HNB}     then $(Y_t)$ is known as the generalized Waring process  \cite{burrell1988modelling,burrell1988predictive,zografi2001generalized}.
\end{remark}
\begin{example}\label{P1-Meix}
Let $(\xi_t)_{t\ge 0}$ be the (symmetric) hyperbolic secant process, with the univariate distributions given by  \eqref{HM} with $\theta=0$.
For $\theta\in(\theta_0,\theta_1)=(-\pi,\pi)$, we get $\kappa(\theta)=-\log \cos^2(\theta/2)$ and $\kappa'(\theta)=\tan(\theta/2)$.
The variance function \eqref{Var-Funct} is  $V(m)=(1+m^2)/2$.

The corresponding  L\'evy process $(X_t^{(\theta)})$  has the marginal density of $X_t^{(\theta)}$ given by
\begin{equation}\label{HM}
f(x;t,\theta) = \frac{(2\cos(\tfrac \theta 2))^{2t}}{2\pi \Gamma(2 t)}  |\Gamma(t+ ix)|^2 e^{\theta x},\; t>0.
\end{equation}

Let $\Theta$ be a random variable with values in $(-\pi,\pi)$ and let
$(Y_t)_{t\ge 0}$ be the corresponding Markov process.
 (This process  was studied in \cite{grigelionis1999processes}.)

 Suppose that $\tan(\Theta/2)$ is integrable and that  there are constants
$p,r$ such that $(Y_t+p)/(t+r)$ is a martingale with respect to natural filtration, $t\geq 0$. Then  $\Theta$ has the following distribution:
 \begin{equation}
    \label{Hip-alpha}
  h(d\theta)=C   \left( \cos (\tfrac\theta2)\right)^{2r}e^{p \theta}\1_{(-\pi,\pi)}(\theta)d\theta,
  \end{equation}
     where   $C=C(p,r)$ is the normalizing constant that does not depend on $\theta$.

Under \eqref{Hip-alpha}, if  $r>1/2$ then
\begin{equation*}
\E(\kappa'(\Theta))=\tfrac{p}{r},\; \Var(\kappa'(\Theta))=\frac{p^2+r^2}{r^2(2r-1)}.
\end{equation*}
(Here we used \eqref{Amazing} to compute the variance.)
It is easy to see that \eqref{Asume1},  \eqref{Asume2} hold for $p\in\RR$, $r>1/2$.
\end{example}
\begin{remark}\label{R.5.1}
It is known that $\E(X_t^{(\theta)})=t \tan \left(\tfrac{\theta}{2}\right)$ and $\Var(X_t^{(\theta)})=\frac{t}{2} \sec ^2\left(\frac{\theta}{2}\right)$.  (To see this, differentiate \eqref{HM} with respect to $\theta$ and integrate the answer with respect to $x$.)
\end{remark}

\section{Stitching L\'evy-Meixner processes}
\label{Sect:SC} 
In this section we show that stitching constructions work nicely for L\'evy-Meixner processes. Then the resulting processes are quadratic harnesses that we call bi-Meixner processes.
We first recall a Markov version of the terminology based on \cite{Bryc-Matysiak-Wesolowski-04}.

\begin{definition}\label{Def_1} Let $\TT=(T_0,T_1)\subset(0,\infty)$.
A square-integrable Markov process $\Z=(Z_t)_{t\in \TT}$ is a quadratic harness on $\TT$ if  it fulfills the following requirements:
 \begin{enumerate}
 \item $\Z$ is a harness  on $\TT$ with  the first two moments given by %
  \begin{equation}\label{EQ:cov}
\E(Z_t)=0,\: \E(Z_sZ_t)=\min\{s,t\}, \;0\leq s\leq t<\infty.
\end{equation}
\item
 there exist numerical constants
 $\ceta,\ctheta\in\RR$ $\sigma,\tau\geq 0$ and $ \gamma\leq1+2\sqrt{\sigma\tau}$ such that
for all $s<t<u$,

\begin{multline}\label{EQ:q-Var}
\V [Z_t|Z_s,Z_u]
= F_{t,s,u}\left( 1+\ceta \frac{uZ_s-sZ_u}{u-s} +\ctheta\frac{Z_u-Z_s}{u-s}\right. \\
+ \sigma
\frac{(uZ_s-sZ_u)^2}{(u-s)^2} \left.
+\tau\frac{(Z_u-Z_s)^2}{(u-s)^2}
-(1-\gamma)\frac{(Z_u-Z_s)(uZ_s-sZ_u)}{(u-s)^2} \right),
\end{multline}
where
\begin{equation*}\label{Ftsu}
F_{t,s,u}=\frac{(u-t)(t-s)}{u(1+s\sigma)+\tau-s\gamma}.
\end{equation*}
\end{enumerate}
\end{definition}
After centering and standardization, Meixner processes from Section \ref{Sect:Meinxer-Examples} are quadratic harnesses with parameters  $\gamma=1$, $\ceta=\sigma=0$. According to \cite{Wesolowski93}, they are uniquely determined by the remaining two parameters $\beta\in\RR,\tau\geq 0$:  the Wiener process  is a quadratic harness with  $\ctheta=\tau=0$; the (centered and scaled) Poisson process  is a quadratic harness with   $\tau=0$, $\ctheta\ne 0$, the (centered) negative binomial process is a quadratic harness with   $\ctheta^2>4\tau>0$ (elliptic case),  the (centered) gamma process  is a quadratic harness with   $\ctheta^2=4\tau>0$ (parabolic case), and the (centered) hyperbolic secant  process  is a quadratic harness with   $\ctheta^2<4\tau$ (hyperbolic case).

In this section we show that stitching of such processes results in bi-Meixner processes,  defined as quadratic harnesses with parameters such that
$0\leq\sigma\tau<1$, $\gamma=1+2\sqrt{\sigma\tau}$, and $\ceta\sqrt{\tau}=\ctheta\sqrt{\sigma}$. Examples of such processes are the bi-Poisson process \cite[Proposition 4.1]{Bryc-Wesolowski-05} and the bi-Pascal process \cite[Proposition 5.1]{Maja:2009}.
In \cite[Proposition 2.7]{Bryc-Wesolowski-09} we established   that only bi-Meixner processes   may result from  stitching together quadratic harnesses with parameters $\sigma=\ceta=0$.

The following result confirms the latter and gives explicit stitching construction of all bi-Meixner processes.


\begin{theorem}\label{T-main}
Fix a L\'evy-Meixner process $(\xi_t)$ corresponding to variance function \eqref{QV}.
Let $\Theta\in(\theta_0,\theta_1)$ be a non-degenerate random variable with the law $h(d\theta)$  given by \eqref{Exponential-law} with parameters $r>a$, $p\in\RR$ such that $\kappa'(\Theta)$ is square integrable, and such that   \eqref{Asume1}, \eqref{Asume2} hold,   as listed in the third column of Table \ref{Table2}.
With %
$v^2=\Var(\kappa'(\Theta))>0$, consider process $(Z_t)$
defined in \eqref{Exp-Z}.

Then $(Z_t)$ is a quadratic harness on $(0,\infty)$.
With $m=p/r$, the parameters are
$$\alpha=\beta=\frac{V'(m)}{\sqrt{V(m)(r-a)}}, \; \sigma=\tau=\frac{a}{r-a}, \; \gamma=1+2\sqrt{\sigma\tau}=\frac{r+a}{r-a},$$   see also Table \ref{Table2}.
\end{theorem}

\begin{table}[hbt]
\begin{tabular}{|ccc|c|}
$\ceta=\ctheta$ & $\sigma=\tau$ & domain & L\'evy process $(\xi_t)$\\ \hline
0&0&  $r>0$, $p\in\RR$ & \mbox{ Wiener (Example \ref{P1-Wiener})} \\
&&&\\
$\frac{1}{\sqrt{p}}$&0& $r>0$, $p>0$ & \mbox{ Poisson (Example \ref{P1-Poiss})} \\
&&&\\
$\frac{2}{\sqrt{r-1}}$& $\frac{1}{r-1}$& $r>1$, $p>0$ & \mbox{ gamma (Example \ref{P1-gamma})}\\
&&&\\
$\frac{r+2p}{\sqrt{p(p+r)(r-1)}}$&$\frac{1}{r-1}$ &  $r>1$, $p> 0$ &  \mbox{ negative binomial  (Example \ref{P1-Pasc})}\\
&&&\\
$\frac{2p}{\sqrt{(p^2+r^2)(2r-1)}}$ &$\frac{1}{2r-1} $ & $r>1/2$, $p\in\RR$ & \mbox{ hyperbolic secant (Example \ref{P1-Meix})}
\end{tabular}
\caption{Parameters of quadratic harnesses on $(0,\infty)$,  under appropriate randomization with parameters $r,p$. \label{Table2}}
\end{table}

\section{Proofs}\label{Sect:Proofs}
  \subsection{Stitching Lemma}
The following technical lemma will be used  to stitch together two quadratic harnesses on adjacent intervals.
\begin{lemma}\label{Generic-lemma}
Suppose a square-integrable Markov  $(Z_t)_{t\in(0,\infty)}$ is a harness , and that both  $(Z_t)_{t\in(0,1)}$ and  $(Z_t)_{t\in(1,\infty)}$ are quadratic harness with the same parameters $\ceta,\ctheta,\sigma,\tau,\gamma$.
  If  $\Var(Z_1|Z_s,Z_u)$ is given by the formula \eqref{EQ:q-Var} with $t=1$, and  with the same parameters $\ceta,\ctheta,\sigma,\tau,\gamma$,  then $(Z_t)_{t>0}$  is a quadratic harness on $(0,\infty)$.
\end{lemma}

\begin{proof}
 Denote
\begin{equation}\label{Z2Delta}
\Delta_{s,t}=\tfrac{Z_t-Z_s}{t-s},\quad \widetilde \Delta_{s,t}=\tfrac{tZ_s-sZ_t}{t-s}.
\end{equation}

By time-inversion, it suffices to consider formula \eqref{EQ:q-Var} in the case $s<t<1<u$.  By Markov property,
$$
\Var(Z_t|Z_s,Z_u)=\E\big(\Var(Z_t|Z_s,Z_1)\big|Z_s,Z_u\big)+\Var(\E(Z_t|Z_s,Z_1)|Z_s,Z_u)\,.
$$
Denote the right hand side of \eqref{EQ:q-Var} by $F_{t,s,u}K(Z_s,Z_u)$. Since $\E(Z_t|Z_s,Z_1)$ is given by \eqref{LR},
 \begin{multline*}\label{tst}
 \Var(\E(Z_t|\calF_{s,1})|\calF_{s,u})=\frac{(t-s)^2}{(1-s)^2}\Var(Z_1|Z_s,Z_u)
 \\=\frac{(t-s)^2(u-1)}{(1-s)(u(1+\sigma s)+\tau-\gamma s)} K(Z_s,Z_u)\,.
\end{multline*}
Next, we write
$$
\E\big(\Var(Z_t|\calF_{s,1})\big|\calF_{s,u}\big)=
\frac{(1-t) (t-s)}{s \sigma +\tau +1-s \gamma }\E(K(Z_s,Z_1)|Z_s,Z_u).
 $$
Since the coefficient $F_{t,s,u}$ is determined by integrating both sides of \eqref{EQ:q-Var}, to end the proof, it suffices to show that
$\E(K(Z_s,Z_1)|Z_s,Z_u)$ is a constant multiple of $K(Z_s,Z_u)$, and we do not need to keep track of the constants.
So it remains to show that
\begin{equation}\label{KKK}
\E(K(Z_s,Z_1)|Z_s,Z_u)=C_{s,u} K(Z_s,Z_u)
\end{equation}
 for any $s<1<u$ and some  constant $C_{s,u}$.

 We have
\begin{equation}\label{K2Delta}
K(Z_s,Z_t)=1+\ceta\widetilde \Delta_{s,t}+\ctheta \Delta_{s,t}+\sigma \widetilde \Delta_{s,t}^2+\tau  \Delta_{s,t}^2-(1-\gamma)\widetilde \Delta_{s,t} \Delta_{s,t}.
\end{equation}

It is easy to check that \eqref{LR} implies %
\begin{equation}\label{DDD}
  \E(\Delta_{s,t}|\calF_{s,u})=\Delta_{s,u},\; \E(\widetilde \Delta_{s,t}|\calF_{s,u})=\widetilde\Delta_{s,u}.
\end{equation}

Since $\Var(\Delta _{s,t}|\calF_{s,u})$, $\Var(\widetilde \Delta _{s,t}|\calF_{s,u})$ and $\cov(\Delta _{s,t},\widetilde \Delta _{s,t}|\calF_{s,u})$ are all proportional to $\Var(Z_t|\calF_{s,u})$, see \eqref{Z2Delta}, from  \eqref{DDD}
 we get
\begin{eqnarray*}
 \E(\Delta^2_{s,1}|\calF_{s,u})&=&\Delta_{s,u}^2+\frac{1}{(1-s)^2}\Var(Z_1|\calF_{s,u}) \,,\\
  \E(\widetilde\Delta^2_{s,1}|\calF_{s,u})&=&\widetilde\Delta_{s,u}^2+\frac{s^2}{(1-s)^2}\Var(Z_1|\calF_{s,u})\,, \\
  \E(\Delta_{s,1}\widetilde\Delta_{s,1}|\calF_{s,u})&=&\Delta_{s,u}\widetilde\Delta_{s,u}-\frac{s}{(1-s)^2}\Var(Z_1|\calF_{s,u})\,.
\end{eqnarray*}
By assumption (ii), $\Var(Z_1|\calF_{s,u})$ is proportional to $K(Z_s,Z_u)$. Using \eqref{K2Delta}, from these formulas together with   \eqref{DDD} we get
$$
\E(K(Z_s,Z_1)|Z_s,Z_u)=K(Z_s,Z_u)+ \frac{\tau+\sigma s^2+(1-\gamma)s}{(1-s)^2} K(Z_s,Z_u),
$$
which proves \eqref{KKK}.
\end{proof}

 \subsection{Proof of Theorem \ref{T-main}}\label{Sect:StpoPp}
The proof consists of series of Claims which verify the assumptions of Lemma \ref{Generic-lemma}.
Some steps do not rely on the specific law of $\Theta$ and in such cases we denote $m=\E(\kappa'(\Theta))$ and $v^2=\Var(\kappa'(\Theta))$.

 We first prove an auxiliary formula.
\begin{lemma}
 If $\xi_t$ is square-integrable, \eqref{Asume2}  holds and
$\Theta$ has law \eqref{Exponential-law} with $p,r$  such that $\kappa'(\Theta)$ is square integrable, then
\begin{equation}\label{Amazing}
\E(\Var(Y_s|\Theta))={s}{r}\Var(\kappa'(\Theta)).
\end{equation}
\end{lemma}
\begin{proof}
Integrating by parts and using \eqref{Asume2} we get \begin{multline*}
C\int_{\theta_0}^{\theta_1}\kappa''(\theta) e^{p\theta -r \kappa(\theta)}d\theta
=C \kappa'(\theta) e^{p\theta -r \kappa(\theta)}\Big|_{\theta=\theta_0}^{\theta=\theta_1}-
\int_{\theta_0}^{\theta_1}\kappa'(\theta)\left(p-r\kappa'(\theta)\right) h(d\theta)
\\=
r\int_{\theta_0}^{\theta_1}(\kappa'(\theta))^2 h(d\theta)
-p\int_{\theta_0}^{\theta_1}\kappa'(\theta) h(d\theta) =
r \Var(\kappa'(\Theta)),
\end{multline*}
where in the last step we used \eqref{p/r}. Since $\Var(X_s^{(\theta)})=s\kappa''(\theta)$,
we get \eqref{Amazing}.
\end{proof}
Next we identify the covariance of the stitched process.
\begin{proposition}\label{P2.4}
If $(\xi_t)$ is square-integrable, \eqref{Asume2} holds and
$\Theta$ has law \eqref{Exponential-law} with $p\in\RR$, $r>0$
such that $\kappa'(\Theta)$ is square integrable, then with $v^2= \Var(\kappa'(\Theta))$, the stitched process $(Z_t)$ has covariance \eqref{EQ:cov}.
\end{proposition}

\begin{claim}\label{Claim1} For $s<t$ we have
\begin{equation}\label{miniY}
\cov(Y_s,Y_t)=  v^2 s(t+r)
\end{equation}
Thus $\Cov(Z_s, Z_t)=s$ if $s<t<1$ or if $1<s<t$.
\end{claim}
\begin{proof}   From
\begin{equation*}\label{G:cov}
\Cov(Y_s,Y_t)=\E(\Cov(Y_s,Y_t|\Theta))+\Cov(\E(Y_s|\Theta),\E(Y_t|\Theta)),
\end{equation*}
 we see that for $s\leq t$, $\Cov(Y_s,Y_t)=s v^2 t+  \E\Var(Y_s|\Theta)$, and the formula follows from \eqref{Amazing}.

 Using \eqref{Exp-Z}, from \eqref{miniY} we compute  $\Cov(Z_s, Z_t)$ on $(0,1)$ and on $(1,\infty)$.
\end{proof}
\begin{proof}[Proof of Proposition \ref{P2.4}]
By Claim \ref{Claim1},  the covariance is as required for $0\leq s<u<1$ and   for $1<s<u$, so by time-reversibility argument it remains only to consider the case $s\leq 1<u$.

Since by the law of large numbers $Y_t/t\to \kappa'(\Theta)$ in mean square as $t\to\infty$, we     have $Z_1=\lim_{s\to 1-}Z_s$  in mean square.
Therefore, we only need to consider the case  $s<1<u$.
The argument here does not depend on the specific law of $\Theta$. Using notation \eqref{rprmt}, from  \eqref{Exp-Z}
 we get
$$
\cov(Z_s,Z_u)=\frac{(1-s)(u-1) }{r^2v^2}\cov(Y_{s'},Y'_{u'})\,.$$
By conditional independence
$$\cov(Y_{s'},Y'_{u'})=  \cov(\E(Y_{s'}|\Theta),\E(Y'_{u'}|\Theta))=s'u' \Var(\kappa'(\Theta))= \frac{ v ^2 r^2 s}{(1-s)(u-1)}.$$
So $\cov(Z_s,Z_u)=\min\{s,u\}$ and  \eqref{EQ:cov} holds.
\end{proof}

\begin{claim}\label{Cm}  $(Z_t)_{t\in(0,1)}$  and $(Z_t)_{t\in(1,\infty)}$ are quadratic harnesses with  the same parameters $ \ceta= \ctheta$, $\sigma=\tau$,
 $\gamma=1+2\sqrt{\sigma\tau}$,
as specified in Theorem \ref{T-main}. 
\end{claim}
\begin{proof}
 For $0<s<t<u<1$,  the conditional law of $Y_t$ under the bivariate conditioning is the same as the conditional  law of $\xi_t$, see \eqref{Inverse-trans}.
 So $\Var(Y_t|Y_s,Y_u) $ is a quadratic function of $Y_s,Y_u$, as determined by the underlying L\'evy-Meixner process $(\xi_t)$.
  Specifically, for a L\'evy process generating \NEF\ with quadratic variance function \eqref{QV}, we have
  \begin{equation}\label{miniVY}
 \Var(Y_t|Y_s,Y_u)=\frac{(t-s)(u-t)}{u-s+a}V\left(\frac{Y_u-Y_s}{u-s}\right).
 \end{equation}
 Indeed, $\Var(\xi_t|\xi_s,\xi_u)=\Var(\xi_t-\xi_s|\xi_u-\xi_s)$ is given by the same expression as $\Var(\xi_{t-s}|\xi_{u-s}) $.
 To end the proof, we use the Laplace transform to show that for $t<u$,
 \begin{equation}\label{miniVV}
 \Var(\xi_{t}|\xi_{u})=\frac{t(u-t)}{u+a}V(\tfrac{\xi_u}{u}).
\end{equation}
 Differentiating the joint Laplace transform
 $\E\left(\exp (z_1\xi_t+z_2\xi_u)\right)=\exp(t\kappa(z_1+z_2)+(u-t)\kappa(z_2))$
 at $z_1=0$ we get
 \begin{eqnarray}
\label{t2}
   \E \left(\xi_t^2e^{z_2\xi_u}\right)&=&t\kappa''(z_2) e^{u\kappa(z_2)}+ t^2 (\kappa'(z_2))^2 e^{u\kappa(z_2)}.
\\
\label{z}
   \E\left( \xi_ue^{z_2\xi_u}\right)&=&u\kappa'(z_2) e^{u\kappa(z_2)}.
\\
\label{z2}
  \E \left(\xi_u^2e^{z_2\xi_u}\right)&=&u\kappa''(z_2) e^{u\kappa(z_2)}+ u^2 (\kappa'(z_2))^2 e^{u\kappa(z_2)}.
\end{eqnarray}
  Since $\kappa''(z_2)=V(\kappa'(z_2))=a (\kappa'(z_2))^2+b\kappa'(z_2)+c$, we can use \eqref{z} and \eqref{z2} to express $\kappa'(z_2) e^{u\kappa(z_2)}$ and $(\kappa'(z_2))^2 e^{u\kappa(z_2)}$ in terms of
  $\E\left( \xi_u e^{z_2\xi_u}\right)$ and $  \E\left( \xi_u^2e^{z_2\xi_u}\right)$.
  Inserting these expressions into right hand side of \eqref{t2}, we get
$$
\E\left( \left(\xi_t^2-\frac{t^2}{u^2}\xi_u^2\right)e^{z_2\xi_u}\right)=\frac{t(u-t)}{u+a}\E\left(\left(a\frac{\xi_u^2}{u^2}+b\frac{\xi_u}{u}+c\right)e^{z_2\xi_u}\right)
.$$
Since $E(\xi_t|\xi_u)=t \xi_u/u$, it is well known that  the last identity implies \eqref{miniVV}, see  \cite[Section 1.1.3]{Kagan-Linnik-Rao:1973}. This proves \eqref{miniVV} and hence \eqref{miniVY} follows.

Once we have \eqref{miniVY}, the reasoning is elementary.
Using notation \eqref{stu},
\begin{multline*}
\Var(Z_t|Z_s,Z_u)=\frac{(1-t)^2}{r^2v^2}\Var(Y_{t'}|Y_{s'},Y_{u'}) 
=
 \frac{(1-t)^2(t'-s')(u'-t')}{r^2v^2(u'-s'+a)}V\left(\frac{Y_{u'}-Y_{s'}}{u'-s'}\right).
\end{multline*}
Noting that
 $$
 \frac{Y_{u'}-Y_{s'}}{u'-s'}=\frac{p}{r}+v\frac{Z_u-Z_s}{u-s}+v\frac{uZ_s-sZ_u}{u-s},
 $$
 we verify that
$(Z_t) $ is a quadratic harness on $(0,1)$ with parameters 
$ \ceta= \ctheta$, $\sigma=\tau$,
 $\gamma=1+2\sqrt{\sigma\tau}$,
as specified in Theorem \ref{T-main}. (The calculation is omitted.)

 Quadratic harness property on $(1,\infty)$ is a consequence of time-inversion, and the parameters are preserved, as $ \ceta= \ctheta$, $\sigma=\tau$.
\end{proof}

\begin{claim}\label{EW|YY'-exp}
If $(\xi_t)$ is a L\'evy-Meixner process with variance function   \eqref{QV} and $\Theta$ has  distribution  \eqref{Exponential-law}, then
$\Var(Z_1|Z_s,Z_u)$ is given by the formula \eqref{EQ:q-Var} with $t=1$, and  with the  parameters $\ceta,\ctheta,\sigma,\tau,\gamma$ as specified in Theorem \ref{T-main}. \end{claim}
\begin{proof}  Using the notation \eqref{rprmt}, we show that if the variance function $V(m)$ is a quadratic expression, then $\Var(\kappa'(\Theta)|Y_{s'},Y'_{u'})$ is a quadratic polynomial in $Y_{s'},Y'_{u'}$.

As previously, let $\varphi(x,y)$ be a bounded measurable function. Recall \eqref{mumu}. Using \eqref{mu-phi0} and integration by parts with respect to $\theta$, we get
\begin{multline*}
\iiint\left(\left(\kappa'(\theta)\right)^2-\frac{(x+z+p)^2}{(u'+s'+r)^2}\right)\varphi(x,y)\mu(dx,d\theta,dz)
\\=
\iiint \kappa'(\theta)\left(\kappa'(\theta)-\frac{x+z+p}{u'+s'+r}\right)\varphi(x,y)\mu(dx,d\theta,dz)
\\
=\tfrac{C}{u'+s'+r}\iint \varphi(x,y)\Big[\kappa'(\theta) e^{ (x+z+p)\theta-(u'+s'+r)\kappa(\theta)}\Big|_{\theta=\theta_0}^{\theta=\theta_1} \\+ \int_{\theta_0}^{\theta_1}\kappa''(\theta) e^{ (x+z+p)\theta-(u'+s'+r)\kappa(\theta)}d\theta \Big] g_{s'}(dx) g_{u'}(dz).
\end{multline*}
Since \eqref{Asume2} gives
$$
\kappa'(\theta) e^{ (x+z+p)\theta-(u'+s'+r)\kappa(\theta)}\Big|_{\theta=\theta_0}^{\theta=\theta_1}=0
$$
for all admissible $x,z$, we get
\begin{multline}\label{min-id}
\iiint\left(\left(\kappa'(\theta)\right)^2-\frac{(x+z+p)^2}{(u'+s'+r)^2}\right)\varphi(x,y)\mu(dx,d\theta,dz)
\\=\tfrac{C}{u'+s'+r} \iint \varphi(x,y) \int_{\theta_0}^{\theta_1}V(\kappa'(\theta) ) e^{ (x+z+p)\theta-(u'+s'+r)\kappa(\theta)}d\theta  g_{s'}(dx) g_{u'}(dz).
\end{multline}
Now we use \eqref{QV} to convert \eqref{min-id} into an equation for  $\iiint\left(\kappa'(\theta)\right)^2\varphi(x,y)\mu(dx,d\theta,dz)$. After a calculation (use \eqref{mu-phi0}), from this equation we get
\begin{multline*}
\iiint\left(\left(\kappa'(\theta)\right)^2-\frac{(x+z+p)^2}{(r+s'+u')^2}\right)\varphi(x,y)\mu(dx,d\theta,dz)
\\=\frac{1}{r+s'+u'-a}
\iint \left(a\frac{(x+z+p)^2}{(r+s'+u')^2}+
b \frac{x+z+p}{r+s'+u'}+ c
\right)\varphi(x,y)
 g_{s'}(dx) g_{u'}(dz)
 \\
 =\frac{1}{r+s'+u'-a}
\iint V\left(  \frac{x+z+p}{r+s'+u'}\right)\varphi(x,y)
 g_{s'}(dx) g_{u'}(dz).
\end{multline*}
Thus
$$
\Var((\kappa'(\Theta))^2|Y_{s'},Y'_{u'})
=\frac{1}{r+s'+u'-a} V\left( \frac{Y_{s'}+Y' _{u'}+p}{r+s'+u'}\right).
$$
From  \eqref{Y2ZZ}, for some non-random constant $C_{s,u}$ we have
$$\Var(Z_1|Z_s,Z_u)=C_{s,u}V\left(\frac{p}{r}+v\frac{Z_u-Z_s}{u-s}+v\frac{uZ_s-sZ_u}{u-s}\right),$$
 matching the parameters that we already got for the case $0<s<t<u<1$.
\end{proof}
\begin{proof}[Proof of Theorem \ref{T-main}]
From Propositions \ref{L-Markov} and  \ref{P2.3} we know that $(Z_t)$ is a Markov harness. Proposition \ref{P2.4} shows that its covariance is \eqref{EQ:cov}.
Claims \ref{Cm} and \ref{EW|YY'-exp}  show that the assumptions of Lemma \ref{Generic-lemma} are satisfied, so $(Z_t)$ is a quadratic harness on $(0,\infty)$ with parameters  as specified in Theorem \ref{T-main}.

\end{proof}

\subsection*{Acknowledgement}
We would like to thank Persi Diaconis for a discussion that lead us to reference \cite{Diaconis:1979}.
This research was partially supported by NSF
grant \#DMS-0904720, and by the Taft Research Center.

\bibliographystyle{plain}
\bibliography{../Vita,bim-2010,Mobius_08_cpy}

\end{document}